\documentclass{amsart}
\newtheorem{theorem}{Theorem}

\newtheorem{lemma}[theorem]{Lemma}
\newtheorem{conjecture}[theorem]{Conjecture}
\newcommand{\bigO}[1]{\mathcal O\left(#1\right)}

\title[Seeking a quadratic refinement]%
 {Seeking a quadratic refinement\\of Sendov's conjecture}

\author{Michael J. Miller}

\address{Department of Mathematics, Le Moyne College,
Syracuse, New York 13214, USA}

\email{millermj@lemoyne.edu}

\dedicatory{This paper is dedicated to Sandy Segal, who introduced
me to Sendov's conjecture.}

\date{5-Sep-2025}

\thanks{Thanks to Le Moyne College for providing the author with substantial
research computing resources.}



\begin{document}

\begin{abstract}
A conjecture of Sendov states that if a polynomial has all its roots in the
unit disk and if $\beta$ is one of those roots, then within one unit of $\beta$
lies a root of the polynomial's derivative.  If we define $r(\beta)$ to be the
greatest possible distance between $\beta$ and the closest root of the
derivative, then Sendov's conjecture claims that $r(\beta) \le 1$.

In this paper, we conjecture that there is a constant $c>0$ so that $r(\beta)
\le 1-c\beta(1-\beta)$ for all $\beta \in [0,1]$.  We find such constants for
complex polynomials of degree $2$ and $3$, for real polynomials of degree $4$,
for all polynomials whose roots lie on a line, for all polynomials with
exactly one distinct critical point, and when $\beta$ is sufficiently close 
to $1$.  In addition, we show that experimental data suggests that 
$c\approx0.233$.
\end{abstract}

\maketitle

\section{Introduction}\label{section_1} 

In 1958, Sendov conjectured that if a polynomial (with complex
coefficients) has all its roots in the unit disk, then within one unit
of each of its roots lies a root of its derivative.  Papers by
Sendov \cite{Sendov} and Schmeisser~\cite{Schmeisser} and books by
Sheil-Small \cite[Chapter 6]{Sheil-Small} and Rahman and Schmeisser
\cite[Section 7.3]{Rahman-Schmeisser} summarize the work that has been
done on this conjecture, identifying more than 80 related papers that
have been published in the past 67 years.  Despite this substantial
body of work, Sendov's conjecture has been verified only for special
cases.

Progress on Sendov's conjecture has slowed in the past few decades, as it
appears that some fundamentally new ideas may be needed.  In this paper, we
begin work on a refinement of Sendov's conjecture, with the hope of providing
new pathways toward achieving a solution.

Let $\beta$ be a complex number of modulus at most~$1$.  Define $S(\beta)$ to
be the set of polynomials of degree at least $2$ with complex coefficients, all
roots in the closed unit disk and at least one root at $\beta$.  For a
polynomial~$P \in S(\beta)$, define $d(P, \beta)$ to be the distance between
$\beta$ and the closest root of the derivative~$P'$.  Note that by the
Gauss-Lucas Theorem \cite[Theorem 2.1.1]{Rahman-Schmeisser} all roots of $P'$
are also in the closed unit disk, and so each $d(P, \beta)\le 2$. 

Finally, define $r(\beta)=\sup \{ d(P, \beta): P \in S(\beta) \}$ and note that
$r(\beta) \le 2$.  In this notation, Sendov's conjecture claims simply that
$r(\beta) \le 1$.  To date, the best such bound known to be true is an
improvement of the work of Bojanov, Rahman and Szynal that implies $r(\beta)\le
1.0753829$ \cite[Theorem~7.3.17]{Rahman-Schmeisser}.

In calculating $r(\beta)$, we will assume without loss of generality (by
rotation) that $0 \le \beta \le 1$.  Define $r_n(\beta)=\sup \{ d(P, \beta): P
\in S(\beta) \hbox{ and } \deg P =n \}$.  Bojanov, Rahman and Szynal have shown
\cite[Lemma~4 and $p(z)=z^n-z$]{Bojanov-Rahman-Szynal} that
$r_n(0)=(1/n)^{1/(n-1)}$, so letting $n$ tend to infinity gives $r(0) = 1$.  In
addition, Rubinstein has shown \cite[Theorem~1]{Rubinstein} that each
$r_n(1)=1$, so $r(1)=1$.  Given that $r(\beta)=1$ at both endpoints of the
interval $0 \le \beta \le 1$, the best possible linear (in $\beta$) bound on
$r(\beta)$ is that $r(\beta) \le 1$, which is the claim of Sendov's conjecture.

To preserve the bounds of $1$ at $\beta=0$ and $\beta=1$, any
quadratic bound on $r(\beta)$ must be of the form $r(\beta) \le
1-c\beta(1-\beta)$ for some constant $c$.  Note that Sendov's conjecture
asserts that we may take $c=0$.  We refine this with
\begin{conjecture}\label{conjecture_1}
   There is a constant $c>0$ so that $r(\beta) \le 1-c\beta(1-\beta)$ for all
   $\beta \in [0,1]$.
\end{conjecture}

In section~\ref{section_2}, we investigate Conjecture \ref{conjecture_1}
for polynomials of low degree.  In section~\ref{section_3}, we examine 
three other special cases.  In section~\ref{section_4}, we present several 
conjectures, including an estimate for the constant $c$ in Conjecture 
\ref{conjecture_1}.  Finally in section~\ref{section_5}, we look at an 
implication of this conjecture.

\section{Results for small $n$}\label{section_2} 

We begin our search for the constant $c$ in Conjecture \ref{conjecture_1} with

\begin{lemma}\label{lemma_2}
   For every $\beta \in [0,1]$, we have $r_n(\beta) \le
   1-\frac{4-n}{4}\beta(1-\beta)$.
\end{lemma}
\begin{proof}
By a result of Schmeisser \cite[Theorem 7.3.16]{Rahman-Schmeisser} we have
\begin{align*}
r_n(\beta)  &\le \frac{n+2\beta-\beta^2(n-2)}{n+2-\beta(n-2)}\\
            &= 1+(1-\beta)\frac{(n-2)\beta-2}{n+2-\beta(n-2)}\\
            &= 1+(1-\beta)\left[\frac{(n-4)\beta}{4}
                 -\frac{(1-\beta)[8+(n-2)(n-4)\beta]}{4[n+2-\beta(n-2)]}\right]\\
            &\le 1+(1-\beta)\frac{(n-4)\beta}{4}
\end{align*}
and we are done.
\end{proof}

Now define $\displaystyle c_n(\beta)=\frac{1-r_n(\beta)}{\beta(1-\beta)}$ and 
$c_n=\inf \{c_n(\beta) : 0 < \beta < 1\}$.
We know this infimum exists since by Lemma \ref{lemma_2}, $c_n(\beta)$
is bounded below by $(4-n)/4$.  Note that $c_n$ is the largest number 
such that $r_n(\beta) \le 1-c_n\beta(1-\beta)$ for all $\beta \in [0,1]$.

\begin{theorem}\label{theorem_3}
We have $c_2=1/2$ and $c_3=1/3$.
\end{theorem}
\begin{proof}
The result for $c_2$ follows trivially from the fact that
$r_2(\beta)={(1+\beta)}/2$. For polynomials of degree 3, Rahman has shown
\cite[Theorem 2]{Rahman} that
$r_3(\beta)=\big(3\beta+\sqrt{12-3\beta^2}\big)/6$.  Since 
\[
\sqrt{12-3\beta^2} \le \sqrt{12-3\beta^2 +(1/4)(1-\beta^2)^2} 
= 3+(1/2)(1-\beta^2),
\
\]
then
\[
   c_3(\beta) = \frac{1-r_3(\beta)}{\beta(1-\beta)}
              = \frac{6-3\beta-\sqrt{12-3\beta^2}}{6\beta(1-\beta)} 
              \ge \frac{5-\beta}{12\beta} \ge \frac{1}{3}.
\]
In addition, we have (via L'Hospital's rule)
\[ 
   \lim_{\beta\to1} c_3(\beta)
   = \lim_{\beta\to1} \frac{6-3\beta-\sqrt{12-3\beta^2}}{6\beta(1-\beta)} 
   = \frac{1}{3}
\]
and we are done.
\end{proof}

We now examine real polynomials of degree 4.

\begin{lemma}\label{lemma_4}
   For every monic real polynomial $P \in S(\beta)$ of degree $4$ with
   $d(P,\beta)>(1+\beta)/2$, we have $|P'(\beta)| \le (1+\beta)^2$.
\end{lemma}
\begin{proof}
Write $P(z)=(z-\beta)\prod_{i=1}^3 (z-z_i)$ and note that
$P'(\beta)=\prod_{i=1}^3 (\beta-z_i)$.

If $P$ has a root (say) $z_1$ in the half-plane $\{z: \Re(z) > \beta \}$, 
then $|z_1-\beta| \le 1$.  For $i\in\{2, 3\}$, we have $|z_i-\beta| \le
1+\beta$, and so $|P'(\beta)|=\prod_{i=1}^3 |\beta-z_i|\le(1+\beta)^2$
and we are done.

Assume then that the half-plane $\{z: \Re(z) \le \beta \}$ contains
all roots of $P$ and hence also (by the Gauss-Lucas theorem
\cite[Theorem 2.1.1]{Rahman-Schmeisser}) all roots of~$P'$. Since by
hypothesis $d(P,\beta)>(1+\beta)/2$, then any real roots of $P'$ would
be in the interval $[-1, (\beta-1)/2)$.  Now $P$ is a real polynomial
of even degree with a real root $\beta$, hence $P$ has another real
root (say) $z_1 \ge -1$.  Since $P'$ has no real roots in the interval
$[(\beta+z_1)/2,\infty)$ and since $P'(x)>0$ for large real $x$, then
$P'((\beta+z_1)/2)>0$.  Now
\[
   P'(z)=(z-\beta)(z-z_1)(2z-z_2-z_3)+(2z-\beta-z_1)(z-z_2)(z-z_3),
\]
so 
\[
   0 < P'\left(\frac{\beta+z_1}{2}\right)
   = -\frac{(\beta-z_1)^2}{4}(\beta+z_1-z_2-z_3),
\]
and so we have $z_2+z_3 > \beta+z_1 \ge \beta-1$.  Now $z_2$ and $z_3$
cannot both be real, else the interval $((\beta-1)/2,\beta]$ would
contain the larger of the two as well as $\beta$, and hence by Rolle's
theorem a real root of $P'$.  Thus $z_2$ and $z_3$ must be complex
conjugates with $\Re(z_2)=\Re(z_3)>(\beta-1)/2$, so for $i\in\{2, 3\}$ we
have
\[
   |z_i-\beta|^2 = z_i \bar z_i - 2\beta\Re(z_i) + \beta^2 
   \le 1 -2\beta(\beta-1)/2 + \beta^2 = 1+\beta.
\]
Note that $|z_1-\beta| \le 1+\beta$, and so $|P'(\beta)|=\prod_{i=1}^3
|\beta-z_i|\le(1+\beta)^2$ and we are done.
\end{proof}

We mention in passing that Lemma \ref{lemma_4} may fail for nonreal
polynomials, as can be seen by choosing $\beta=0.674$ and
\[
P(z)=\int_\beta^z 4(w+0.24-0.38i)(w+0.13+0.25i)^2\,dw.
\]
A numerical calculation establishes that the roots of $P$ have
moduli less than $1$, so $P \in S(\beta)$.  However, $d(P,\beta)
\approx 0.842 > 0.837 = (1+\beta)/2$, but $|P'(\beta)| \approx 2.807 >
2.802 \approx (1+\beta)^2$.

Finally, we prove Conjecture \ref{conjecture_1} for all real polynomials
of degree 4 with
\begin{theorem}\label{theorem_5}
   For every real polynomial $P \in S(\beta)$ of degree $4$ we have
   $d(P, \beta) \le 1-(1/3)\beta(1-\beta)$.
\end{theorem}
\begin{proof}
If $d(P, \beta) \le (1+\beta)/2$, then 
$d(P, \beta) \le (1+\beta)/2+(1/6)(1-\beta)(3-2\beta)
   =1-(1/3)\beta(1-\beta)$ 
and we are done.  Assume then that $d(P,\beta)>(1+\beta)/2$, assume 
(without loss of generality) that $P$ is monic and write 
$P'(z)=4\prod_{i=1}^3 (z-\zeta_i)$.  Then using Lemma
\ref{lemma_4} we have
\begin{equation}\label{eqn2.2}
   4(d(P, \beta))^3 \le 4\prod_{i=1}^3 |\zeta_i-\beta| = |P'(\beta)|
   \le (1+\beta)^2.
\end{equation}

Now $(1+\beta)^2\le (1+\beta)^2+3(1-\beta)^2=4(1-\beta(1-\beta))$.
In addition, $1-x \le 1-x+(x^2/27)(9-x)=(1-x/3)^3$ for $0 \le x \le 1$,
so letting $x=\beta(1-\beta)$ gives us
\begin{equation}\label{eqn2.3}
   (1+\beta)^2 \le 4(1-\beta(1-\beta)) \le 4(1-\beta(1-\beta)/3)^3.
\end{equation}

Combining lines \ref{eqn2.2} and \ref{eqn2.3} gives us that $4(d(P,
\beta))^3 \le 4(1-\beta(1-\beta)/3)^3$ and our result follows.
\end{proof}

\section{Special cases}\label{section_3} 

We next show that Conjecture \ref{conjecture_1} is true provided that
$\beta$ is sufficiently close to~1 (where ``sufficiently close''
depends on the degree of the polynomial), using

\begin{theorem}\label{theorem_6}
   For every integer $n \ge 2$, if $\beta$ is sufficiently close to
   $1$, then we have $r_n(\beta) \le 1-(3/10)\beta(1-\beta)$.
\end{theorem}
\begin{proof}
Given Theorem \ref{theorem_3}, we may assume that $n \ge 4$.  Note that
\[
   1-(3/10)(1-\beta)+(3/10)(1-\beta)^2 =
   1-(3/10)\beta(1-\beta).
\]

If $n \ne 5$, then Miller has shown \cite[Theorem 1 and part 6 of
Lemma 8, using $n \ne 4$ and $c_{n+1}=D_1+D_2/n$]{Miller} that
$r_n(\beta) = 1+c_n(1-\beta)+\mathcal O(1-\beta)^2$ with $c_n <
-3/10$, so $r_n(\beta)<1-(3/10)(1-\beta)$ when $\beta$ is sufficiently
close to $1$ and the result follows.

If $n=5$, then Miller has shown \cite[Theorem 1, using $n=4$,
$D_1=-1/5$, $D_2=-2/5$, $D_3=0$, $D_4=0$, $D_5=1/25$ and
$D_6=-2/25$]{Miller} that 
\[
   r_5(\beta) = 1-(3/10)(1-\beta) +(1/200)(1-\beta)^2
   +\mathcal O(1-\beta)^3
\]
and the result follows.
\end{proof}

We now show that Conjecture \ref{conjecture_1} is true for polynomials
with all roots on a line, via
\begin{theorem} \label{theorem_7}
   If all the roots of $P \in S(\beta)$ lie on a line, then 
 $d(P, \beta) \le 1-(1/2)\beta({1-\beta})$.
\end{theorem}
\begin{proof}
Let $P\in S(\beta)$ be a polynomial of degree $n$ with all roots on a line.
Write $P(z)=(z-\beta)Q(z)$ and let $z_1,\dots,z_{n-1}$ be the roots of $Q$
and $\zeta_1, \dots, \zeta_{n-1}$ be the roots of the derivative $P'$.  
Since the roots of $P$ lie on a line, then (by a rotation around $\beta$) 
Rolle's Theorem implies that the roots of $P'$ also lie on that line, 
alternating with the roots of $P$.

If $P$ had a multiple root at $\beta$ then $d(P, \beta)=0$ and the theorem 
would be trivially true, so we may assume that every $z_i\ne\beta$.

Consider the roots and critical points of $P$ that lie to one side of $\beta$
on the line, and number them (by increasing distance from $\beta$) as
$\beta, \zeta_1, z_1, \zeta_2, z_2, \dots, \zeta_k, z_k$.  Note that
$|\zeta_1-\beta|\ge d(P, \beta)$ and $|z_k-\beta|\le1+\beta$ and 
$|\zeta_{i+1}-\beta| \ge |z_i-\beta|$.

Then
\begin{equation}\label{eqn_9.1}
   \prod_{i=1}^k \frac{|\zeta_i-\beta|}{|z_i-\beta|} 
    = |\zeta_1-\beta| 
        \left(\prod_{i=1}^{k-1} \frac{|\zeta_{i+1}-\beta|}{|z_i-\beta|}\right) 
             \frac{1}{|z_k-\beta|}
    \ge \frac{d(P, \beta)}{|z_k-\beta|}.
\end{equation}

Since $P'(\beta)=Q(\beta)$ then 
$n\prod_{i=1}^{n-1} |\zeta_i-\beta| = \prod_{i=1}^{n-1} |z_i-\beta|$.

If all the roots of $Q$ lie to one side of $\beta$, then by inequality
\ref{eqn_9.1} we have
\[
   \frac{1}{n}= \prod_{i=1}^{n-1} \frac{|\zeta_i-\beta|}{|z_i-\beta|} 
     \ge \frac{d(P, \beta)}{|z_{n-1}-\beta|}
     \ge \frac{d(P, \beta)}{1+\beta}
\] 
so 
\[
d(P, \beta)\le \frac{1+\beta}{n} 
  \le \frac{1+\beta}{2} + \frac{(\beta-1)^2}{2} = 1-(1/2)\beta(1-\beta).
\]

If the roots of $Q$ lie on both sides of $\beta$, then by applying inequality
\ref{eqn_9.1} separately to each side, we get
\[
   \frac{1}{n}= \prod_{i=1}^{n-1} \frac{|\zeta_i-\beta|}{|z_i-\beta|} 
     \ge \frac{[d(P, \beta)]^2}{|z_k-\beta||w_m-\beta|},
\] 
where $z_k$ and $w_m$ are the roots of $Q$ furthest from $\beta$ on opposite 
sides of the line.  Then Euclid's ``Intersecting Chords Theorem'' implies that
$|z_k-\beta||w_m-\beta|\le (1-\beta)(1+\beta)$ so
\[
d(P, \beta) \le \sqrt\frac{1-\beta^2}{n} 
   \le \sqrt{\frac{1-\beta^2}{2} + \frac{\beta^2+2(\beta-1)^2}{4}}
   = 1-\frac{\beta}{2} \le 1-(1/2)\beta(1-\beta).
\]

\end{proof}

We now show that Conjecture \ref{conjecture_1} is true for polynomials
with exactly one distinct critical point, via

\begin{theorem}\label{theorem_8}
For every polynomial $P \in S(\beta)$ with exactly one distinct critical 
point, $d(P, \beta) \le 1-(1/3)\beta(1-\beta)$.
\end{theorem}
\begin{proof}
Take any $P \in S(\beta)$ of degree $n$ with exactly one distinct critical 
point at~$\zeta$.  Then all roots of $P$ are on the circle centered at
$\zeta$ with radius $r=|\beta-\zeta|$, so $d(P, \beta)=r$. 

If $n$ is even, then the maximum value of $r$ is achieved when the root
of $P$ furthest from $\beta$ is $-1$, and then $r=(1+\beta)/2$.  Thus  
in general we have
\[
r \le \frac{1+\beta}{2} + \frac{(1-\beta)(3-2\beta)}{6} 
   = 1-(1/3)\beta(1-\beta).
\]

If $n$ is odd, then the maximum value of $r$ is achieved when the
two roots of $P$ furthest from $\beta$ are conjugates of modulus $1$, 
and thus of the form $z_i=\beta-r-re^{\pm i\pi/n}$.  Then

\begin{align*}
1=|z_i|^2 &= [\beta-r-r\cos(\pi/n)]^2+[r \sin(\pi/n)]^2\\
          &= \beta^2-2\beta r(1+\cos(\pi/n))+2r^2(1+\cos(\pi/n))
\end{align*}
which implies $2[1+\cos(\pi/n)](r^2-\beta r) =1-\beta^2$.
Since $n\ge3$, then $\cos(\pi/n)\ge1/2$ and so 
$r^2-\beta r\le (1-\beta^2)/3$.

If $r>(2+\beta)/3$ then
\[
r(r-\beta) > \frac{2+\beta}{3}\cdot\frac{2-2\beta}{3} 
  = \frac{1-\beta^2}{3} + \frac{(1-\beta)^2}{9} \ge \frac{1-\beta^2}{3},
\]
which is a contradiction, so $r\le(2+\beta)/3+(1-\beta)^2/3 
  = 1-(1/3)\beta(1-\beta)$.

\end{proof}

\section{Conjectures}\label{section_4} 

While we have verified Conjecture \ref{conjecture_1} for a number of
special cases, there remain many unanswered questions.  To provide
guidance about potential answers, we present some conjectures
based on the results of extensive experimental searches for polynomials
with minimal $c_n$ values.

\subsection{Polynomials of low degree}\ 

Sendov's conjecture is known to be true for polynomials of degree $n\le8$
\cite{Brown-Xiang}.  In this section, we seek improved bounds for
these polynomials.

Theorem \ref{theorem_3} showed that $c_2=1/2$ and $c_3=1/3$.

For polynomials of degree $4$, $5$, and $8$, experimental results suggest that
$c_n(\beta)$ decreases to its minimum as
$\beta$ approaches $1$.  Assuming this is true,
then using the notation and results of \cite[Theorem 1]{Miller}, we have
$r_{n+1}(\beta)=1+(D_1+D_2/n)(1-\beta)+\bigO{1-\beta}^2$, so then
$c_{n+1}(\beta)=-(D_1+D_2/n)+\bigO{1-\beta}$ for $\beta$ sufficiently close to
$1$, so then $c_{n+1}=-(D_1+D_2/n)$.  In particular, we can calculate
$c_4=-[(-1/4)+(-1/4)/3]=1/3$, $c_5=-[(-1/5)+(-2/5)/4]=3/10$ and
$c_8=-[((\sqrt{2}-2)/2)+((\sqrt{2}-2)/2)/7]=(8-4\sqrt{2})/7$.

For $n=6$ and $7$, experimental results suggest that the minimum $c_n$
are achieved by the polynomials defined by the data in Table \ref{table_1}, 
and from that we can calculate the corresponding values of $r_n$ and $c_n$ 
in Table \ref{table_2}.

\begin{table}[ht]
\caption{Data to define polynomials 
$P(z)=\int_\beta^z \prod_i(z-\zeta_i)^{m_i}\,dz$ }
\label{table_1}
\begin{tabular}{|c|c|p{3.0in}|}
   \hline
   $n$ & $\beta$ & critical points with multiplicities $(\zeta_i, m_i)$ \\
   \hline
$6$ & $0.788188270312241$ & $(0.0469833741737209\pm0.576557593047195\,i, 1)$\newline$(-0.150855581784183, 3)$\\ \hline
$7$ & $0.722412690737455$ & $(0.141636085050414\pm0.729946180810592\,i, 1)$\newline$(-0.210391089590075, 4)$\\ \hline
\end{tabular}
\end{table}

\begin{table}[ht]
\caption{Conjectured smallest $c_n$ for polynomials of degree $n$}
\label{table_2}
\begin{tabular}{|c|c|c|}
   \hline
   $n$ & $r_n$ & $c_n$ \\
   \hline
     $4$ & - & $1/3$\\
     $5$ & - & $3/10$\\
     $6$ & $0.939043852096423$ & $0.365121611819106$ \\
     $7$ & $0.932803780327529$ & $0.335088765359222$ \\
     $8$ & - & $(8-4\sqrt{2})/7\approx 0.334735107215374$ \\
  \hline
\end{tabular}
\end{table}

Based on this experimental data, we make

\begin{conjecture}\label{conjecture_9} 
The conjectured values of $c_n$ listed in Table \ref{table_2} are the correct 
values.
\end{conjecture}

Sendov's conjecture is that $c_n\ge0$, so the values in Table \ref{table_2}
substantially improve the bounds on $d(P, \beta)$ provided 
by Sendov's conjecture.  Note that our Theorem \ref{theorem_5} has verified
the value of $c_4$ in Table \ref{table_2} for real polynomials.

\subsection{Polynomials with at most 2 distinct critical points}\ 

Theorem \ref{theorem_8} settles the case when a polynomial has exactly
one distinct critical point.

In this section, we will restrict ourselves to polynomials of degree $n$ 
in $S(\beta)$ with at most $2$ distinct critical points.  

When $n$ is not a multiple of $3$, experimental results suggest that
$c_n(\beta)$ decreases to its minimum as $\beta$ approaches $1$.  
If this is true, then Theorem \ref{theorem_6} shows that for these
polynomials, $d(P, \beta)\le 1-(3/10)\beta(1-\beta)$.

When $n$ is a multiple of $3$, experimental results suggest that the
lowest bound for $c_n$ is approached when $n$ is large.
For specified large values of $n$, the data defining polynomials of 
degree $n$ with the smallest experimentally achieved values of $c_n$
 are listed in Table \ref{table_3}.

\begin{table}[ht]
\caption{Data to define polynomials 
$P(z)=\int_\beta^z \prod_i(z-\zeta_i)^{m_i}\,dz$ }
\label{table_3}
\begin{tabular}{|r|c|p{3.0in}|}
   \hline
   $n$ & $\beta$ & critical points with multiplicities $(\zeta_i, m_i)$ \\
   \hline

$300$ & $0.999999387871706$ & $(0.0000399954026754105
\newline{}\hfill+0.00898912895921542\,i, 1)$
\newline$(-4.07186589509001E{-}7
\newline{}\hfill-0.0000301636393125034\,i, 298)$\\ \hline

$600$ & $0.999999923978711$ & $(0.0000100349982627836
\newline{}\hfill+0.00449123594189648\,i, 1)$
\newline$(-5.06247934490226E{-}8
\newline{}\hfill-7.51035212032506E{-}6\,i, 598)$\\ \hline

$1200$ & $0.999999990527972$ & $(2.51325659135395E{-}6
\newline{}\hfill+0.00224480126064511\,i, 1)$
\newline$(-6.31118508088983E{-}9
\newline{}\hfill-1.87378597659023E{-}6\,i, 1198)$\\ \hline

$2400$ & $0.999999998817902$ & $(6.28876830399965E{-}7
\newline{}\hfill+0.00112219836751760\,i, 1)$
\newline$(-7.87846617611952E{-}10
\newline{}\hfill-4.67972335535035E{-}7\,i, 2398)$\\ \hline

\end{tabular}
\end{table}

From the polynomials in Table \ref{table_3}, calculating $r_n$ 
and $c_n$ results in Table \ref{table_4}.  From this table, 
we predict that $c_n\ge1/3$ for all $n$ and that
$c_n=(1/3)+O(1/n)$ and so when $n$ is a multiple of $3$ we have
$d(P, \beta)\le 1-(1/3)\beta(1-\beta)$.

Putting these two cases together, we can make
\begin{conjecture}\label{conjecture_10}\
If $P \in S(\beta)$ has at most $2$ distinct critical points, then
$d(P, \beta)\le 1-(3/10)\beta(1-\beta)$.
\end{conjecture}

\begin{table}[ht]
\caption{Calculated values}
\label{table_4}
\begin{tabular}{|r|c|c|l|}
   \hline
   $n$ & $r_n$ & $c_n$ & $n(c_n-\frac{1}{3})$ \\
   \hline

$300$ & $0.999999795513218$ & $0.334058904562927$ & $0.217671$ \\
$600$ & $0.999999974631707$ & $0.333699878082601$ & $0.219927$ \\
$1200$ & $0.999999996840912$ & $0.333517557970814$ & $0.22107$ \\
$2400$ & $0.999999999605858$ & $0.333425685271983$ & $0.221645$ \\

  \hline
\end{tabular}
\end{table}

\subsection{Polynomials with at most 3 distinct critical points}\ 

In this section, we will restrict ourselves to polynomials of degree $n$ 
in $S(\beta)$ with at most $3$ distinct critical points.

Experimental results suggest that the lowest bound for $c_n$ is approached
when $n$ is large. For specified large values of $n$, 
the data defining polynomials of degree~$n$ with the smallest experimentally 
achieved values of $c_n$ are listed in Table \ref{table_5}.

\begin{table}[ht]
\caption{Data to define polynomials 
$P(z)=\int_\beta^z \prod_i(z-\zeta_i)^{m_i}\,dz$ }
\label{table_5}
\begin{tabular}{|r|l|p{3.0in}|}
   \hline
   $n$ & $\beta$ & critical points with multiplicities $(\zeta_i, m_i)$ \\
   \hline

$400$ & $0.996978954167785$ & $(0.185748374494393\pm0.583342877793831\,i, 1)$
\newline$(-0.00221270235763630, 397)$\\ \hline

$800$ & $0.998502579477915$ & $(0.185435674615939\pm0.581482907206863\,i, 1)$
\newline$(-0.00109742210050623, 797)$\\ \hline

$1600$ & $0.999250885199129$ & $(0.186470916321556\pm0.582227523214433\,i, 1)$
\newline$(-0.000549178112052197, 1597)$\\ \hline

$3200$ & $0.999625340827601$ & $(0.186990144893956\pm0.582601186890055\,i, 1)$
\newline$(-0.000274706495037257, 3197)$\\ \hline

\end{tabular}
\end{table}

\begin{table}[ht]
\caption{Calculated values}
\label{table_6}
\begin{tabular}{|r|l|l|l|}
   \hline
   $n$ & $r_n$ & $c_n$ & $n(c_n-\frac{4}{15})$ \\
   \hline

$400$ & $0.999191656525421$ & $0.268381534797037$ & $0.685947$ \\
$800$ & $0.999600001578421$ & $0.267525574656365$ & $0.687126$ \\
$1600$ & $0.999800063311181$ & $0.267097347966552$ & $0.68909$ \\
$3200$ & $0.999900047322638$ & $0.266882935081154$ & $0.692059$ \\

  \hline
\end{tabular}
\end{table}

From the polynomials in Table \ref{table_5}, calculating $r_n$ 
and $c_n$ results in Table \ref{table_6}.  From this table, 
we predict that $c_n\ge4/15$ for all $n$ and that
$c_n=(4/15)+O(1/n)$ and so we can make 

\begin{conjecture}\label{conjecture_11}\
If $P \in S(\beta)$ has at most $3$ distinct critical points, then
$d(P, \beta)\le 1-(4/15)\beta(1-\beta)$.
\end{conjecture}

\subsection{Polynomials with at most 4 distinct critical points}\ 

In this section, we will restrict ourselves to polynomials of degree $n$ 
in $S(\beta)$ with at most $4$ distinct critical points.

Experimental results suggest that the lowest bound for $c_n$ is approached
when $n$ is large.  For specified large values of $n$, 
the data defining polynomials of degree~$n$ with the smallest experimentally 
achieved values of $c_n$ are listed in Table  \ref{table_7}.

\begin{table}[ht]
\caption{Data to define polynomials 
$P(z)=\int_\beta^z \prod_i(z-\zeta_i)^{m_i}\,dz$ }
\label{table_7}
\begin{tabular}{|r|l|p{3.0in}|}
   \hline
   $n$ & $\beta$ & critical points with multiplicities $(\zeta_i, m_i)$ \\
   \hline

$401$ & $0.992437606116294$ & $(0.323796567775001\pm0.741084178859921\,i, 1)$
\newline$(-0.00478291485722727
\newline{}\hfill\pm0.0428699295532890\,i, 199)$\\ \hline

$801$ & $0.996200558165452$ & $(0.327047737914208\pm0.741870195976750\,i, 1)$
\newline$(-0.00240066386653551
\newline{}\hfill\pm0.0305365991065751\,i, 399)$\\ \hline

$1601$ & $0.998083040087457$ & $(0.330441008376128\pm0.743851319703156\,i, 1)$
\newline$(-0.00120969122802383
\newline{}\hfill\pm0.0217877367297296\,i, 799)$\\ \hline

$3201$ & $0.999040939079070$ & $(0.331221301238603\pm0.744007489122246\,i, 1)$\newline$(-0.000605196555280278
\newline{}\hfill\pm0.0154212861549267\,i, 1599)$\\ \hline

\end{tabular}
\end{table}

\begin{table}[ht]
\caption{Calculated values}
\label{table_8}
\begin{tabular}{|r|l|l|l|}
   \hline
   $n$ & $r_n$ & $c_n$ & $n(c_n-0.24483)$ \\
   \hline

$401$ & $0.998141572278504$ & $0.247618545888762$ & $1.118207$ \\
$801$ & $0.999068007959806$ & $0.246232640425885$ & $1.123515$ \\
$1601$ & $0.999530223821050$ & $0.245533799424740$ & $1.126783$ \\
$3201$ & $0.999765078683669$ & $0.245184462572637$ & $1.134635$ \\

  \hline
\end{tabular}
\end{table}

From the polynomials in Table \ref{table_7}, calculating $r_n$ 
and $c_n$ results in Table \ref{table_8}.  From this table, 
we predict that there is a $c\approx 0.24483$ so that $c_n\ge c$ 
for all $n$ and that $c_n=c+O(1/n)$ and so we can make 

\begin{conjecture}\label{conjecture_12}\
There is a $c\approx 0.24483$ so that 
if $P \in S(\beta)$ has at most $4$ distinct critical points, then
$d(P, \beta)\le 1-c\beta(1-\beta)$.
\end{conjecture}

\subsection{General polynomials}

In this section, we will consider all polynomials of degree $n$ 
in $S(\beta)$.

Experimental results suggest that the lowest bound for $c_n$ is approached
when $n$ is large.  For specified large values of $n$, 
the data defining polynomials of degree $n$ with the smallest experimentally 
achieved values of $c_n$ are listed in Table \ref{table_9}.  Note that the 
search space for these polynomials is extremely large, so the values in 
this table may be subject to error.

\begin{table}[ht]
\caption{Data to define polynomials 
$P(z)=\int_\beta^z \prod_i(z-\zeta_i)^{m_i}\,dz$ }
\label{table_9}
\begin{tabular}{|r|l|p{3.0in}|}
   \hline
   $n$ & $\beta$ & critical points with multiplicities $(\zeta_i, m_i)$ \\
   \hline

$50$ & $0.932492785695482$ & $(0.288272070152277\pm0.742813688052441\,i, 1)$
\newline$(-0.00780080468460338\pm0.287472554138991\,i, 5)$
\newline$(-0.0507632875087811, 37)$\\ \hline

$100$ & $0.952685415849498$ & $(0.354063712751498\pm0.786842134763372\,i, 1)$
\newline$(0.0710215926019638\pm0.447367177199525\,i, 3)$
\newline$(-0.0337143539272170
\newline{}\hfill\pm0.0669625465717506\,i, 45)$
\newline$(-0.0359846446150283, 1)$\\ \hline

$200$ & $0.964050955456579$ & $(0.410070382621719\pm0.822404449690337\,i, 1)$
\newline$(0.187251759456666\pm0.616300708293278\,i, 2)$
\newline$(0.0205767627249383\pm0.305122928666932\,i, 6)$
\newline$(-0.0216922035881849\pm0.107489433645153\,i, 20)$
\newline$(-0.0275354270381114, 141)$\\ \hline

$400$ & $0.978173100606026$ & $(0.432580453239963\pm0.831987241184806\,i, 1)$\newline$(0.238399354695915\pm0.665288592429133\,i, 2)$\newline$(0.0507055817517486\pm0.360108466749262\,i, 6)$\newline$(-0.0112482862151233\pm0.104496055691722\,i, 35)$\newline$(-0.0167510705518637, 311)$\\ \hline

  \hline
\end{tabular}
\end{table}

\begin{table}[ht]
\caption{Calculated values}
\label{table_10}
\begin{tabular}{|r|l|l|l|}
   \hline
   $n$ & $r_n$ & $c_n$ & $n(c_n-0.233)$ \\
   \hline

$50$ & $0.983256073204263$ & $0.265987758064115$ & $1.649388$ \\
$100$ & $0.988670060464527$ & $0.251352406649079$ & $1.835241$ \\
$200$ & $0.991586382494691$ & $0.242770226189940$ & $1.954045$ \\
$400$ & $0.994924171157889$ & $0.237738329523279$ & $1.895332$ \\

  \hline
\end{tabular}
\end{table}

From the polynomials in Table \ref{table_9}, calculating $r_n$ 
and $c_n$ results in Table \ref{table_10}.  From this table, 
we predict that there is a $c\approx 0.233$ so that $c_n\ge c$ 
for all $n$ and that $c_n=c+O(1/n)$ and so we can make 

\begin{conjecture}\label{conjecture_13}\
There is a $c\approx 0.233$ so that if $P \in S(\beta)$ then
$d(P, \beta)\le 1-c\beta(1-\beta)$.
\end{conjecture}

\section{Implications}\label{section_5} 

Phelps and Rodriguez have conjectured \cite[after
Theorem~5]{Phelps-Rodriguez} that the only monic polynomials $P \in
S(\beta)$ with $d(P, \beta) \ge 1$ are of the form $z^n-e^{it}$ for
some~$t$.  Now Bojanov, Rahman and Szynal have shown \cite[Lemma
4]{Bojanov-Rahman-Szynal} that $d(P,0)<1$ for every $P \in S(0)$, so
if $P \in S(\beta)$ with $d(P, \beta) \ge 1$, then our
Conjecture~\ref{conjecture_1} would imply that $\beta=1$.  Given this,
Rubinstein has shown \cite[Theorem 1]{Rubinstein} that
$P(z)=c(z^n-e^{it})$.  Thus our Conjecture~\ref{conjecture_1} implies the
conjecture of Phelps and Rodriguez.


\end{document}